\documentclass[12pt]{amsart}

\usepackage{setspace}
\usepackage{geometry}
\usepackage{todonotes}
\usepackage{amsfonts, amssymb, amscd}
\numberwithin{equation}{section}
\usepackage[symbol]{footmisc}

\usepackage{bm}
\usepackage{verbatim}
\usepackage{mathrsfs}
\usepackage{graphicx}
\usepackage{tikz-cd}
\usepackage{subcaption}
\usepackage{listings}
\usepackage{subfiles}
\usepackage[toc,page]{appendix}
\usepackage{mathtools}
\usepackage{comment}
\usepackage{enumerate}
\usepackage{enumitem}
\usepackage[all]{xy}
\usepackage{graphicx}
\graphicspath{{images/}}
\usepackage{appendix}
\usepackage[colorlinks=true, citecolor=green, linkcolor=blue, filecolor=magenta, urlcolor=red,backref=page]{hyperref}

\usepackage{pgfplots}
\pgfplotsset{compat=1.17}

\newcommand{\bQ}{\mathbb{Q}}
\newcommand{\bR}{\mathbb{R}}
\newcommand{\bZ}{\mathbb{Z}}

\newcommand{\bN}{\mathbb{N}}
\newcommand{\bP}{\mathbb{P}}

\newcommand{\cA}{\mathcal{A}}

\newcommand{\cE}{\mathcal{E}}
\newcommand{\cF}{\mathcal{F}}
\newcommand{\cG}{\mathcal{G}}
\newcommand{\cH}{\mathcal{H}}
\newcommand{\cI}{\mathcal{I}}

\newcommand{\cL}{\mathcal{L}}
\newcommand{\cM}{\mathcal{M}}
\newcommand{\cN}{\mathcal{N}}
\newcommand{\cO}{\mathcal{O}}
\newcommand{\cP}{\mathcal{P}}
\newcommand{\cQ}{\mathcal{Q}}

\newcommand{\cU}{\mathcal{U}}

\newcommand{\cX}{\mathcal{X}}
\newcommand{\cY}{\mathcal{Y}}

\newcommand{\qlin}{\sim_{\mathbb{Q}}}

\newcommand{\Pic}{\mathrm{Pic}}

\newcommand{\Supp}{\mathrm{Supp}}

\newcommand{\vol}{\operatorname{vol}}

\newcommand{\mult}{\operatorname{mult}}
\theoremstyle{plain} 
\newtheorem{thm}{Theorem}[section] 
\newtheorem{lemma}[thm]{Lemma}
\newtheorem{prop}[thm]{Proposition}
\newtheorem{cor}[thm]{Corollary}
\theoremstyle{definition} 
\newtheorem{defn}[thm]{Definition} 

\newtheorem{exam}[thm]{Example}

 \newtheorem{rem}[thm]{Remark}

\begin{document}
	
	\title{Failure of  Boundedness for Generalised Log Canonical Surfaces}
	
	\author{Christopher Hacon and Xiaowei Jiang}
    \address{Department of Mathematics, University of Utah, 155 S 1400 E, Salt Lake City, Utah, 84112, USA}
\email{hacon@math.utah.edu}
	
	\address{Yau Mathematical Sciences Center, Tsinghua University, Hai Dian District, Beijing, 100084, China}
	\email{jxw20@mails.tsinghua.edu.cn}
	\date{\today}
	
	\thanks{}

	\begin{abstract}
		In this paper,
        we construct counterexamples to the boundedness of generalised log canonical models of surfaces with fixed appropriate invariants, where the underlying varieties can have arbitrary Kodaira dimension. 
        This answers a question of Birkar and the first author. 
		
	\end{abstract}
	
	\maketitle
	\tableofcontents
	\section{Introduction}
	Throughout this paper,  we  work over an algebraically closed field $k$ of characteristic zero.
	\vspace*{10pt}

According to the minimal model conjecture and the abundance conjecture, any variety $Y$ with mild singularities is birational to a normal projective variety $X$ such that either $X$ admits a Mori fiber space $X \to Z$, or $K_X$ is semiample. Classification of such $X$  means construction of the corresponding moduli spaces. Usually, the first step to construct moduli spaces is to show that, after fixing certain invariants, our varieties belong to a bounded family.

When $K_X$ is semiample, there exists a contraction $f: X \to Z$ to a normal variety $Z$.  The canonical bundle formula
$$ K_X \sim_{\mathbb{Q}} f^*(K_Z + B_Z + M_Z) $$
shows that the structure of $(Z, B_Z + M_Z)$ plays a fundamental role in understanding the geometry of $X$.
We can then regard $(Z, B_Z + M_Z)$ as a \textit{generalised pair} with ample $K_Z + B_Z + M_Z$, that is, a \textit{generalised log canonical (lc) model}.
We refer the reader to $\S$\ref{S:g-pair} for  background on  generalised pairs and their singularities.
In this paper, we study the boundedness of generalised lc models. 
For the definition of boundedness for generalised  pairs, see $\S$\ref{S:bdd family}.

\begin{defn}

        Fix $d\in \mathbb{N}$, $\Phi\subset \bQ^{\ge 0}$ a DCC set, and $v\in \bQ^{>0}$. 
Let $\mathcal{F}_{glc}(d,\Phi,v)$ be the set of projective generalised pairs $(X,B+M)$ 
with data $X'\overset{\phi}\to X$ and $M'$ where 
\begin{itemize}
\item $(X,B+M)$ is generalised lc of dimension $d$, 
\item the coefficients of $B$ are in $\Phi$, 
\item $M'=\sum \mu_iM_i'$ where $M_i'$ are nef Cartier and $\mu_i\in \Phi$,  
\item $K_X+B+M$ is ample, and 
\item  
$
\vol(K_X+B+M)=v.
$ 
\end{itemize}
    
\end{defn}

When $M=0$, that is, when $(X,B)$ is a usual pair, $\mathcal{F}_{glc}(d,\Phi,v)$ forms a bounded family by \cite{haconBirationalAutomorphismsVarieties2013,haconACCLogCanonical2014,haconBoundednessModuliVarieties2018}, which is crucial for the construction of moduli spaces of varieties of general type \cite{kollarFamiliesVarietiesGeneral2023}.

When $(X,B+M)$ is generalised klt, boundedness is known by \cite{birkarBoundednessVolumeGeneralised2021} and has applications to studying the (birational) boundedness and moduli of klt good minimal models \cite{jiaoBoundednessPolarizedCalabiYau2022,zhuBoundednessStableMinimal2025,jiangBoundednessModuliTraditional2023}.

Although $\mathcal{F}_{glc}(d,\Phi,v)$ is log birationally bounded by \cite[Proposition 5.2]{birkarBoundednessVolumeGeneralised2021}, Birkar and  the first author construct an unexpected counterexample to its boundedness for $d\geq 3$, see \cite[\S 5.3]{birkarVariationsGeneralisedPairs2022}. Nevertheless,  the generalised pairs given by the canonical bundle formula are bounded under certain conditions \cite[Theorem 1.3]{birkarVariationsGeneralisedPairs2022}, which plays a key role in the study of boundedness and moduli of slc good minimal models \cite{birkarModuliAlgebraicVarieties2022}.

It is clear that  $\mathcal{F}_{glc}(1, \Phi, v)$  forms a bounded family, and it is natural to ask whether  $\mathcal{F}_{glc}(2, \Phi, v)$  is bounded or not \cite[Question 5.1]{birkarVariationsGeneralisedPairs2022}. In \cite{filipazziBoundednessLogCanonical2018}, Filipazzi shows that $\mathcal{F}_{glc}(2, \Phi, v)$ is bounded under the assumption that the Cartier index of $M$ (not that of $M'$) is bounded.
In this paper, we will construct examples showing that $\mathcal{F}_{glc}(2, \Phi, v)$ is not bounded in general, where $X_i\in \mathcal{F}_{glc}(2, \Phi, v)$ can have arbitrary Kodaira dimension. 

\begin{thm}\label{thm: unbdd}
Fix $\kappa\in\{-\infty,0,1,2\}$. Let $\cP_{\kappa}\subset\cF_{glc}(2,\Phi,v)$ be the subset of generalised pairs such that $\kappa(X)=\kappa$ for every $(X,B+M)\in \cP_{\kappa}$.
Then $\cP_{\kappa}$ is not bounded.
\end{thm}

In \cite{birkarVariationsGeneralisedPairs2022}, Birkar and the first author construct a set of generalised pairs $(X_i, B_i + M_i)$ in $\mathcal{F}_{glc}(3, \Phi, v)$ such that each $X_i$ is an lc Fano variety and the Cartier index of $K_{X_i}$ is unbounded. However, after a flip, $X_i$ becomes klt and bounded, while the failure of boundedness in dimension 2 implies that we cannot expect $\mathcal{F}_{glc}(d, \Phi, v)$ to be bounded in codimension one. On the other hand, for generalised lc surfaces $(X_i, B_i + M_i)$ in $\mathcal{F}_{glc}(2, \Phi, v)$, it can be shown that the Cartier index of $K_{X_i}$ is bounded. Furthermore, if $X_i$ is an lc Fano surface, then $X_i$  is bounded.

\begin{thm}\label{thm:bdd}
Let $\cQ\subset \cF_{glc}(2,\Phi,v)$ be a subset of generalised pairs. Then $\cQ$ is bounded if one of the following holds for every $(X,B+M)\in \cQ$:
\begin{enumerate}
\item $X$ has rational singularities,
\item $X\to Z$ is a minimal ruled fibration onto a nonsingular curve $Z$,
\item $-K_{X}$ is ample,
\item $K_{X}$ is ample.
\end{enumerate}
\end{thm}
By \textit{minimal ruled fibration}, we mean a fibration whose general fiber is a nonsingular rational curve, with no exceptional curves of the first kind contained in any fiber. 
While boundedness in the Fano case fails in dimension 3 by \cite{birkarVariationsGeneralisedPairs2022}, boundedness in the canonically polarized case holds in any dimension, see Corollary \ref{cor:bdd can model}.

    

\subsection*{Acknowledgement}
This work was done when the second
author visited the first author at the University of Utah. The second author would like to thank the University of Utah for its hospitality.
He also thanks his advisor Caucher Birkar for  constant support and encouragement, and Qingyuan Xue  for helpful comments.
The first author  was partially supported by NSF research grants no: DMS-2301374 and by a grant from the Simons Foundation SFI-MPS-MOV-00006719-07.

	\section{Preliminaries}

\subsection{Generalised pairs and singularities}\label{S:g-pair}

We refer reader to \cite{birkarGeneralisedPairsBirational2021} for a survey of  generalised pairs.
	A \textit{generalised pair} $(X,B,M)$ consists of
	\begin{itemize}
		\item a normal variety $X$,
		\item an effective $\bR$-divisor $B$ on $X$, and
		\item a b-$\bR$-Cartier b-divisor $M$ over $X$, represented by a projective birational morphism $f: X'\to X$ and an $\bR$-Cartier $\bR$-divisor $M'$ on $X'$
	such that $M'$ is nef  and $K_X+B+M$ is $\bR$-Cartier, where $M:=f_*M'$.  
 \end{itemize}
 We will often refer to such a generalised pair by saying that $(X, B + M)$ is a generalised pair with data $X'\to X$ and $M'$.

	Let $D$ be a prime divisor over $X$. Replace $X'$ with a log resolution of $(X,B)$ such that  $D$ is a prime divisor on $X'$. We can write 
	$$K_{X'}+B'+M'=f^*(K_X+B+M).$$ 
	We define the \textit{generalised log discrepancy} of  $D$ to be $a(D,X,B,M)=1-\mult_D B'$.
	
	We say that $(X,B+M)$ is \textit{generalised klt} (resp. \textit{generalised lc}) if $a(D,X,B,M)>0$ (resp. $a(D,X,B,M)\geq0$) for every prime divisor $D$ over $X$.   

\begin{rem}
   If $(X,B+M)$ is a generalised log canonical surface, 
   then the negativity lemma \cite[Lemma 3.39]{kollarBirationalGeometryAlgebraic1998} applied to $M'$ and the numerical pullback of $M$ implies that $(X,B)$ is numerically log canonical.
Thus, $(X,B)$ is log canonical by \cite[Proposition 3.5]{fujinoMinimalModelTheory2012}, and $M$ is $\bQ$-Cartier.
Moreover, by \cite[Remark 3.6]{filipazziBoundednessLogCanonical2018}, $M$ is nef.
Similarly, we can show that $X$ is log canonical.
     
\end{rem}
    
        \subsection{Bounded family}\label{S:bdd family}
A \textit{couple} $(X, D)$ consists of a normal projective variety $X$ and a reduced divisor $D$ on $X$. We do not assume that $K_X + D$ is $\bQ$-Cartier or that $(X,D)$ has nice singularities.

   A set $\cP$ of generalised pairs is said to be a \textit{bounded family} if there is a fixed $r\in \bN$ such that for any $(X,B+M)\in \cP$,  we can find a very ample divisor $A$ on $X$ satisfying
   $$ A^{\dim X} \leq r \quad \text{and} \quad A^{\dim X - 1} \cdot (K_X + B + M) \leq r. $$

The first condition implies that the underlying variety $X$ is bounded. 
If the coefficients of $B$ belong to a DCC set $\Phi$, then the first two conditions imply that $(X, \text{Supp } B)$ belongs to a \textit{log bounded family of couples}. This is equivalent to saying that
there is a quasi-projective scheme $\cX$, a reduced divisor $\cE$ on $\cX$, and a projective morphism  $h: \cX\to T$, 
where $T$ is of ﬁnite type and $\cE$ does not contain any ﬁber, such that for every $(X, B+M) \in\cP$, there is a closed point $t\in T$ and an isomorphism  $f : \cX_t  \to X$ 
such that $\cE_t$ contains the support of $f^{-1}_*B$.

However, since $M$ is not necessarily effective, we cannot control $\text{Supp } M$. In practice, we can only bound the intersection number $A^{\dim X - 1} \cdot M$.

    \subsection{Positivity of line bundles on the blow ups of surfaces}
Firstly, we recall some results on ampleness and very ampleness for blow ups of $\bP^2$.

     \begin{thm}[{\cite[Theorem 2.4]{biancofioreHyperplaneSectionsBlowUps1989}}]\label{thm:arbitrary blow up P2}

    Let $Y \to\bP^2$ be the blow up of $\bP^ 2$ at $r$ distinct points $p_ 1 ,\ldots, p_ r$ with exceptional divisors $E_ 1 ,\ldots,
		E_ r$. 
		Let $H$ denote the pull-back of $\cO_{ \bP^ 2} (1)$. 
        For $d,m_1,\ldots,m_r\in \bZ^{> 0}$,
		 let $L = dH - \sum_{i=1}^{r}m_i E_ i$ be a line bundle on $Y$. 
		 If $$d\geq 1+\sum_{i=1}^rm_i,$$
         then $L$ is very ample.
\end{thm}

The bound is sharp and can  be improved only if not all points $p_1, \ldots, p_r$ lie on a line.  
If $d < 1 + \sum_{i=1}^r m_i$, determining whether $L$ is very ample requires studying the positions of the points $p_1, \ldots, p_r$.
    
	\begin{thm}[{\cite[Theorem 2.1]{hanumanthuPositivityLineBundles2017}}]\label{thm:ample on rational}
       		Let $C\subset \bP^2$ be an irreducible and reduced curve of degree $e$. 
		Let $Y \to\bP^2$ be the blow up of $\bP^ 2$ at $r$ distinct smooth points $p_ 1 ,\ldots, p_ r\in C$ with exceptional divisors $E_ 1 ,\ldots,
		E_ r$.  
		Let $H$ denote the pull-back of $\cO_{ \bP^ 2} (1)$.        
    For $d,m_1,\ldots,m_r\in \bZ^{> 0}$,
		 let $L = dH - \sum_{i=1}^{r}m_i E_ i$ be a line bundle on $Y$. 
		 Let $F$  denote the proper transform of $C$ on $Y$.
	If $$L \cdot F =3d-\sum_{i=1}^rm_i> 0$$ 
    and $$d>m_{i_1}+\cdots+m_{i_e}$$
    for any $e$ distinct indices $i_ 1, \ldots , i_ e\in\{1,\ldots,r\}$, then $L$ is ample.
		
	\end{thm}
	
	\begin{thm}[{\cite[Theorem 3.6]{hanumanthuPositivityLineBundles2017}}]\label{thm:very ample on rational}
        		Let $C\subset \bP^2$ be an irreducible and reduced curve of degree $e$. 
		Let $Y \to\bP^2$ be the blow up of $\bP^ 2$ at $r$ distinct smooth points $p_ 1 ,\ldots, p_ r\in C$ with exceptional divisors $E_ 1 ,\ldots,
		E_ r$.  
		Let $H$ denote the pull-back of $\cO_{ \bP^ 2} (1)$.        
 For $d,m\in \bZ^{> 0}$, let $L = dH -m \sum_{i=1}^{r} E_ i$ be a line bundle on $Y$.
		If $$(d + 3)e > r(m + 1) \quad \text{and} \quad r \geq e^2 + 2 ,$$ then $L$ is very ample.
		
	\end{thm}

We also need the following result on the very ampleness of blow ups of ruled surfaces at arbitrary points.
\begin{thm}[{\cite[Theorem 1.3]{biancofioreHyperplaneSectionsRuled1990}}]\label{thm:positivity on blow up ruled}
Let $W$ be a ruled surface with invariant $d\geq 0$ over a curve $C$ of genus $g$, with a fiber $F$ and a section $C^-$ satisfying $(C^-)^2=-d$.
Let $Y\to W$ be the blow up of $W$ at $r$ distinct points $p_1,\ldots,p_r$ with exceptional divisors $E_1,\ldots,E_r$. For simplicity, we still denote by $C^-$ and $F$ the divisors on $Y$ given by the pullbacks of the divisors $C^-$ and $F$ on $W$, with no confusion.
For $a,b,m_1,\ldots,m_r\in \bZ^{> 0}$, let 
$$L=aC^-+bF-\sum_{i=1}^r m_iE_i$$
be a line bundle on $Y$.
If  
$$b \geq ad + 2g + 1 + \sum_{i=1}^r m_i,$$  
and for any effective curve $C' = F - \sum_{i=1}^r \alpha_i E_i$ with $0 \leq \alpha_i \leq 1$, we have  
$$L \cdot C' = a - \sum_i \alpha_i m_i \geq 1,$$  
then $L$ is very ample.
   
\end{thm}
\subsection{Artin's theorem on contractions}

\begin{thm}[{\cite[Corollary 6.10]{artinAlgebraizationFormalModuli1970}}]
\label{Thm:Arin contraction}
    Let $X'$ be an algebraic space and $Y'\subset X'$ a closed subspace such that $I'=\cI(Y')$ is locally principal, for instance assume that $X'$, $Y'$ are regular and of dimensions $d$, $d-1$, respectively.
    Let $f_0:Y'\to Y$ be a proper map. Suppose that
    \begin{enumerate}
        \item For every coherent sheaf $F$ on $Y'$,  if $n\gg 0$, then
        $$R^1f_{0*}(F\otimes(I'/I'^2)^{\otimes n})=0.$$
        \item For every $n\geq 1$, 
        the canonical map
        $$f_{0*}(\cO_{X'}/I'^n)\otimes_{f_{0*}\cO_{Y'}}\cO_Y\to \cO_Y$$ 
        is surjective.
    \end{enumerate}
    Then there is a modification $f:X'\to X$, $Y\subset X$ whose set-theoretic restriction to $Y$ is $f_0$.

\end{thm}

By \textit{modification}, we mean a pair consisting of a proper map
$f:X'\to X$
of algebraic spaces and a closed subset $Y\subset X$, such that the restriction of $f$ to $U:=X\setminus Y$ is an isomorphism. 
We may also refer to the pair as a \textit{contraction} of $X'$, or as a \textit{dilatation} of $X$, respectively.

Note that condition (2) is automatic if $R^1f_{0*}(I'/I'^2)^{\otimes n}=0$ for every $n\geq 1$.

\section{Proof of Theorem \ref{thm:bdd}}
In this section, we will prove the boundedness of $\mathcal{F}_{glc}(2, \Phi, v)$ under additional assumptions.

We  first recall the birational boundedness of $\mathcal{F}_{glc}(d, \Phi, v)$.
\begin{thm}[{\cite[Proposition 5.2]{birkarBoundednessVolumeGeneralised2021}}]\label{thm:bir bdd}
    Let $d\in \bN$, $\Phi\subset \bQ^{\ge 0}$ be a DCC set, and $v\in \bQ^{>0}$.
Then there exists a bounded set of couples $\mathcal{P}$ such that for each 
$$
(X,B+M)\in \mathcal{F}_{glc}(d,\Phi,v)
$$ 
with data $X'\overset{\phi}\to X$ and $M'=\sum \mu_iM_i'$, there is a log smooth couple $(\overline{X},\overline{\Sigma})\in \mathcal{P}$ and a birational map 
$\overline{X}\dashrightarrow X$ such that 
\begin{itemize}
\item $\overline{\Sigma}\ge \overline{B}$, where $\overline{B}$ is the sum of the reduced exceptional 
divisors of $\pi:\overline{X}\dashrightarrow X$ plus the birational transform of $B$,

\item each $M_i'$  descends to $\overline{X}$, say as $\overline{M}_i$,  and 

\item letting $\overline{M}=\sum \mu_i\overline{M}_i$, we have 
$$
\vol(K_{\overline{X}}+\overline{B}+\overline{M})=v.
$$
\end{itemize}
\end{thm}

\begin{cor}\label{cor:bdd can model}
   Let $d \in \mathbb{N}$, $\Phi \subset \mathbb{Q}^{\geq 0}$ be a DCC set, and $v \in \mathbb{Q}^{>0}$. For each
$$(X, B+M) \in \mathcal{F}_{glc}(d, \Phi, v),$$
 assuming that the MMP conjecture and the Abundance conjecture hold, let $Z$ be the canonical model of $X$.
If $K_X$ is big, then $Z$ is bounded. Moreover, if $K_X$ is ample, then $(X,B+M)$ is bounded.
\end{cor}

\begin{proof}
Let $(\overline{X}, \overline{\Sigma})$ be the birationally bounded model of $(X, B+M)$  given in Theorem \ref{thm:bir bdd}, and let $\overline{E} \subset \overline{\Sigma}$ be the sum of the reduced exceptional divisors of the birational map $\pi:\overline{X} \dashrightarrow X$.
If $K_X$ is big,
then  $Z$ is also the log canonical model of $(\overline{X}, \overline{E})$.
But then since $(\overline{X}, \overline{E})$ is log smooth and it belongs to a bounded family, its log canonical model also belongs to a bounded family, by \cite[Theorem 1.2]{haconBoundednessModuliVarieties2018}.

If $K_X$ is ample, then $X$ belongs to a bounded family. Therefore, there exists $l \in \mathbb{N}$ depending only on $(d, \Phi, v)$ such that $lK_X$ is very ample. 
Then we have
$$(lK_X)^d=l^d\vol(K_X)\leq l^d\vol(K_X + B + M) = l^d v, $$
and 
$$(lK_X)^{d-1} \cdot (K_X + B + M) 
\leq l^{d-1} (K_X + B + M)^d = l^{d-1} v.$$
Hence $(X, B+M)$ is bounded.
\end{proof}

\begin{proof}[Proof of Theorem \ref{thm:bdd}]
We use the notation in Theorem \ref{thm:bir bdd}.

(1). We aim to show that the Cartier index of $M$ is bounded, and hence $(X, B + M)$ is bounded by \cite[Theorem 1.14]{filipazziBoundednessLogCanonical2018}.

By Step 4 of \cite[Proof of Theorem 1.5, p.44]{birkarBoundednessVolumeGeneralised2021}, the coefficients of $B$ and the $\mu_i$ all belong to a fixed finite set depending only on $(2, \Phi, v)$.
By Step 4  of \cite[Proof of Theorem 1.3, p.35]{birkarBoundednessVolumeGeneralised2021},
there is a log smooth couple $(\overline{\cX},\overline{\Omega})$ and a log smooth projective morphism  $(\overline{\cX},\overline{\Omega})\to T$ onto a smooth variety such that for each $(X,B+M)\in\cF_{glc}(2,\Phi,v)$,
there is a closed point $t\in T$ so that $(\overline{X},\overline{\Sigma})\simeq (\overline{\cX}_t,\overline{\Omega}_t)$.
Moreover,  for each $i$, there exist irreducible components $\overline{\cG}_i$, $\overline{\cF}_i$ of $\overline{\Omega}$ such that 
$$\overline{\cF}_{i,t}- \overline{\cG}_{i,t}  \sim n\overline{M}_i$$ for some $n\in \bN$ depending only on $d=2$. Therefore, there exist Cartier divisors $n\overline{\cM}_i$ on $\overline{\cX}$ such that $n\overline{\cM}_{i,t} \sim  n\overline{M}_i$ for each $i$.
Let $\overline{\cM} := \sum \mu_i \overline{\cM}_i$, then $\Supp(\overline{\cM})\subset\overline{\Omega}$. Since $\mu_i$ belongs to a finite set, after replacing $n$ by a bounded multiplier, $n\overline{\cM}$ is Cartier and $n\overline{\cM}_t\sim n\overline{M}$.

We now follow the idea in Step 5 of \cite[Proof of Theorem 1.14, p.845]{filipazziBoundednessLogCanonical2018}.
For each $(X,B+M)\in\cF_{glc}(2,\Phi,v)$ and the birational map $\pi: \overline{X} \dashrightarrow X$,  we can write 
$$\pi^*M = \overline{M} + \sum b_j E_j,$$
where $E=\sum_j  E_j$ is the $\pi$-exceptional divisor.
Since $\overline{M}$ and the exceptional divisors of $\pi: \overline{X} \dashrightarrow X$ deform in the family, and the numerical pullback is computed intersection-theoretically,
it follows from the negative definiteness of  intersection matrix  $(E_j\cdot E_k)_{jk}$  that  the system of linear equations
 $$(\overline{M} + \sum_j b_j E_j) \cdot E_k = 0\ \ \  \text{ for all }k$$
 has a unique solution depending only on $(2,\Phi,v)$.
Therefore, there exists a divisor 
$$\overline{\Phi}\leq \overline{\Omega}+\overline{\cM},$$ supported on $\overline{\Omega}$, such that
$\pi^*M=\overline{\Phi}_t$.  Let $l$ be a positive integer depending only on $(2,\Phi,v)$ such that $l\overline{\Phi}$  is integral. 

Since $X$ has only rational singularities, by \cite[Lemma 4.13]{kollarBirationalGeometryAlgebraic1998}, $l\overline{\Phi}_t\sim 0$ in an analytic neighborhood of any $\pi$-exceptional curve. Thus, $\pi_*\cO_{\overline{\cX}_t}(l\overline{\Phi}_t)$ is locally free in an analytic neighborhood of  any closed point $x\in X$. Then, $lM$ is Cartier at $\widehat{\cO}_{X,x}$ for every closed point $x\in X$. 
Therefore, $lM$ is  Cartier by \cite[Lemma 5.12]{boucksomValuationSpacesMultiplier2015}.

(2). If $X\to Z$ is a minimal ruled fibration onto a nonsingular curve $Z$,   then $X$ has only rational singularities by \cite[Lemma 4.6]{sakaiStructureNormalSurfaces1985}. It follows from (1) that $(X,B+M)$ is bounded.

(3). Since the exceptional divisors of $\pi: \overline{X} \dashrightarrow X$ belong to a bounded family, the dual graphs of the $\pi$-exceptional curves and the corresponding weights are determined. By \cite[Remark 4.9]{kollarBirationalGeometryAlgebraic1998}, the analytic isomorphism type of the germ $x \in X$ is determined by the dual graph. Therefore, the Cartier index of $-K_X$ is bounded.
Then, by the effective base point free theorem \cite[Theorem 1.1]{fujinoEffectiveBasePoint2009} and the very ampleness lemma \cite[Lemma 7.1]{fujinoEffectiveBasepointfreeTheorem2017} for lc pairs, we can find $l\in \bN$ depending only on $(2,\Phi,v)$ such that $-lK_X$ is very ample.

For each $(X,B+M)\in\cF_{glc}(2,\Phi,v)$ and the birational map $\pi: \overline{X} \dashrightarrow X$, we can write 
$$\pi^*K_X = K_{\overline{X}} + \sum a_j E_j$$
and $$\pi^*(K_X+B+M) = K_{\overline{X}} +\pi_*^{-1}B+\overline{M}+ \sum c_j E_j,$$
where $E_j$ are the $\pi$-exceptional divisors.
Moreover, we have 
$$(K_{\overline{X}} + \sum a_j E_j) \cdot E_k = 0 $$
and 
$$(K_{\overline{X}} +\pi_*^{-1}B+\overline{M}+ \sum c_j E_j) \cdot E_k = 0$$
for all $k$. By the same proof as in (1), the coefficients $a_j$ and $c_j$ are determined by $(2, \Phi, v)$.
Since $\pi_*^{-1} B$, $M$, and the exceptional divisors of the birational map $\pi:\overline{X}\dashrightarrow X$ deform in the family,
it follows that $$(-K_X)^2=(-\pi^*K_X)^2$$
and $$(-K_X)\cdot(K_X+B+M)=\pi^*(-K_X)\cdot \pi^*(K_X+B+M)$$ are bounded. Hence, $(X,B+M)$ belongs to a bounded family.

(4). It follows from Corollary \ref{cor:bdd can model}. Alternatively,  to bound $(X, B + M)$, we can bound the Cartier index of $K_X$ and the intersection numbers \[
K_X^2 \quad \text{and} \quad K_X \cdot (K_X + B + M)
\]
by the same argument as in (3).
\end{proof}

	\section{\texorpdfstring{Failure of boundedness for $\cF_{glc}(2,\Phi,v)$}{Failure of boundedness of Fglc(2,Phi,v)}}

In this section, we construct counterexamples to the boundedness of $\cF_{glc}(2, \Phi, v)$, where the underlying varieties $X$ can have arbitrary Kodaira dimension. Since boundedness holds when $X$ has only rational singularities by Theorem \ref{thm:bdd}, in our examples, $X$ has elliptic singularities.
    
	\subsection{Calabi-Yau surfaces}\label{S:K=0}
	In this subsection, we construct an unbounded set of generalised pairs $(X_i, B_i + M_i) \in\cP_0\subset \cF_{glc}(2, \Phi, v)$ for $\Phi = \{0,1
\}$ and $v = 22$, such that $K_{X_i} \sim 0$, $B_i = 0$, and $M_i$ is an ample Weil divisor with an unbounded Cartier index.

Let $C \subset \mathbb{P}^2$ be a smooth elliptic curve, and let $p_{i,1}, \dots, p_{i,11} \in C$ be distinct points such that for each $1 \leq j \leq 11$, the divisor $p_{i,j} - p_0 \in \Pic^0(C)$ has order $n_{i,j}$, where $p_0$ is the identity element of the group structure on $C$. Let $n_i$ be the order of the divisor $\sum_{j=1}^{11} (p_0 - p_{i,j})$.
We may fix the points $p_{i,1}, \dots, p_{i,10}$ and vary $p_{i,11}$ so that $n_{i,j}$ is independent of $i$ for $1 \leq j \leq 10$, and both $n_{i,11}$ and $n_i$ tend to infinity as $i \to \infty$.
          
Let $f_i : Y_i \to \mathbb{P}^2$ be the blow up of $\mathbb{P}^2$ at $p_{i,1}, \dots, p_{i,11}$, and let $E_{i,1}, \dots, E_{i,11}$ be the exceptional divisors over $p_{i,1}, \dots, p_{i,11}$, respectively. Let $F_i$ denote the birational transform of $C$ on $Y_i$. Write $E_i = \sum_{j=1}^{11} E_{i,j}$, then we have
$$ K_{Y_i} = f_i^* K_{\mathbb{P}^2} + E_i \quad \text{and} \quad F_i + E_i = f_i^* C, $$
hence we have
$$ K_{Y_i} + F_i = f_i^* (K_{\mathbb{P}^2} + C) \sim 0. $$
Let $H_i = f_i^* \mathcal{O}_{\mathbb{P}^2}(1)$ and $A_i = 5H_i - E_i$, then by Theorem \ref{thm:very ample on rational}, $A_i$ is very ample on $Y_i$. We have
$$ \text{vol}(A_i) = A_i^2 = 25 - 11 = 14. $$
Hence, $Y_i$ belongs to a bounded family.
	
Since $F_i \sim 3H_i - E_i$, an easy computation shows that $(A_i + 2F_i) \cdot F_i = 0$. Let $\widetilde{p_{i,j}}$ be the points on $F_i$ lying over $p_{i,j}$ for $1 \leq j \leq 11$. Then we have 
		\begin{equation}\nonumber
		\begin{aligned}
			(A_i+2F_i)|_ {F_i}= &   (11H_i-3E_i)|_ {F_i}\\
			=&f_i^*(\cO_{ \bP^ 2}(11))|_{F_i}\otimes\cO_{F_i}(-\sum_{j=1}^{11}\widetilde{p_{i,j}})^{\otimes 3}\\
				=&f_i^*(\cO_{ \bP^ 2}(11)|_{C}\otimes\cO_{C}(-\sum_{j=1}^{11}{p_{i,j}})^{\otimes 3})\\
				=&f_i^*(\cO_{C}(\sum_{j=1}^{11}{(p_0-p_{i,j}}))^{\otimes 3}).	
		\end{aligned}
	\end{equation}
Let $m_i$
	be the order of  $(A_i+2F_i)|_ {F_i}\in \Pic^0(F_i)$ so that $\cO_{Y_i}(m_i(A_i+2F_i))\otimes\cO_{F_i}\simeq \cO_{F_i}$. Then we have $m_i=n_i$ if $3\nmid n_i$,  and $m_i=\frac{n_i}{3}$ if $3\mid n_i$.

\begin{lemma}\label{lem:bpf}
The linear system $|m_i(A_i + 2F_i)|$ is big and base point free on $Y_i$, and it defines a contraction $\pi_i: Y_i \to X_i$ to a projective Calabi-Yau surface, which contracts $F_i$ to a simple elliptic singularity.
\end{lemma}
	\begin{proof}

	    Since $A_i$ is ample and $F_i$ is effective, $A_i + 2F_i$ is big. 
        Moreover,
        we have $(A_i + 2F_i) \cdot F_i' \geq 1$ for any irreducible curve $F_i' \neq F_i$, and $(A_i + 2F_i) \cdot F_i = 0$. Hence, $A_i + 2F_i$ is nef.
        We aim to show  $|m_i(A_i + 2F_i)|$  is base point free.
        
	We  first claim that the restriction map 
	$\Pic(2F_i)\to \Pic(F_i)$ is an isomorphism, which implies
	$$\cO_{Y_i}(m_i(A_i+2F_i))\otimes\cO_{2F_i}\simeq \cO_{2F_i}.$$  
	Indeed, the short  exact sequence
	$$0\to \cO_{Y_i}(-F_i)\otimes\cO_{F_i} \to  \cO_{2F_i}^*\to \cO_{F_i}^*\to 0$$
	 induces a long  exact sequence
	$$H^1(F_i,\cO_{F_i}(-F_i|_{F_i}))\to \Pic(2F_i)\to \Pic(F_i)\to 0.$$
	Since $-F_i^2=2$, we have $H^1(F_i,\cO_{F_i}(-F_i|_{F_i}))=0$, which yields the desired isomorphism.
    
Since $\cO_{Y_i}(m_i(A_i+2F_i))\otimes\cO_{2F_i} \simeq \cO_{2F_i}$, we have the short exact sequence
\begin{equation}
    0\to \mathcal{O}_{Y_i}(m_iA_i+(2m_i-2)F_i) \to \mathcal{O}_{Y_i}(m_i(A_i+2F_i)) \to \mathcal{O}_{2F_i} \to 0.
    \label{eq:exact_seq}
\end{equation}
By Theorem \ref{thm:ample on rational}, the divisor
$$m_iA_i+(2m_i-1)F_i\sim (11m_i-3)H_i-(3m_i-1)E_i$$
is ample. Hence, Kodaira vanishing gives
$$H^1(Y_i, m_iA_i+(2m_i-2)F_i) = H^1(Y_i, K_{Y_i}+m_iA_i+(2m_i-1)F_i) = 0.$$
Therefore, the map
$$H^0(Y_i, m_i(A_i+2F_i)) \to H^0(\cO_{2F_i})$$
induced by the exact sequence (\ref{eq:exact_seq})  is surjective.

Thus, there exists a divisor in the linear system
$|m_i(A_i+2F_i)|$
that does not intersect any point of $F_i = \Supp(2F_i)$.
Moreover, since $A_i$ 
is very ample,  any base point of $|m_i(A_i+2F_i)|$
must be contained in $F_i$.
Therefore, $|m_i(A_i+2F_i)|$ is base point free. Since it is also big, it defines a birational morphism $\pi_i: Y_i \to X_i$ onto a projective surface $X_i$ with $K_{X_i} \sim 0$, which contracts $F_i$ to a simple elliptic singularity.
    	\end{proof}

	Now let $M_{Y_i} := A_i + 2F_i$ be the big and semi-ample Cartier divisor on $Y_i$, which defines the contraction $\pi_i: Y_i \to X_i$ that contracts $F_i$ to a simple elliptic singularity.
    Then we have 
	$$  K_{Y_i}+F_i=\pi_i^*K_{X_i} \quad \text{and} \quad \pi_i^*M_i=M_{Y_i},$$
    where $M_i:=\pi_{i*}M_{Y_i}$.
	Therefore,  $(X_i,M_i)$ is a generalised lc surface with data $Y_i\to X_i$ and $M_{Y_i}$.
	 Moreover, $K_{X_i}+M_i$ is ample and 
	 $$\vol(K_{X_i}+M_i)=\vol(K_{Y_i}+F_i+M_{Y_i})=(11H_i-3E_i)^2=22.$$
	 Thus, 
	 $$(X_i,M_i)\in\cP_0\subset \cF_{glc}(2,\Phi,v)$$
	 for $\Phi=\{0,1\}$ and $v=22$. 
     
     We claim that the Cartier index of $M_i$  is not bounded. Indeed, assume that there exists $m\in \bN$ such that $mM_i$ is Cartier for all $i$. Then, by the effective base point free theorem \cite[Theorem 1.1]{fujinoEffectiveBasePoint2009} and the very ampleness lemma \cite[Lemma 7.1]{fujinoEffectiveBasepointfreeTheorem2017} for lc pairs, we can find $m'\in \bN$ such that $m'M_i$ is very ample for all $i$. Hence, $m'M_{Y_i}$ is big and base point free for all $i$, which is  a contradiction.
	 

	 \begin{prop}\label{prop: M unbdd}
	The set of generalised pairs in $\cP_0$  is not bounded.
	 \end{prop}
	 \begin{proof}
    
   We will show that  there is no fixed $r \in \bN$ and a very ample divisor $N_i$ on $X_i$ satisfying $N_i \cdot (K_{X_i} + M_i) \leq r$ for all $(X_i, M_i) \in \cP_0\subset\cF_{glc}(2, \Phi, v)$.  Thus, $\cP_0$ is not bounded in the sense of \S\ref{S:bdd family}.
   
Let $N_{Y_i} := \pi_i^* N_i$. Then $N_{Y_i}$ is a big and base point free Cartier divisor on $Y_i$ with $N_{Y_i} \cdot F_i = 0$. 
Moreover, for any irreducible curve $F_i' \neq F_i$ on $Y_i$, we have $N_{Y_i} \cdot F_i' > 0$.
Writing $N_{Y_i} \sim d_i H_i - \sum_{j=1}^{11} l_{i,j} E_{i,j}$ for some integers $d_i, l_{i,1}, \dots, l_{i,11}$, the condition $N_{Y_i} \cdot F_i = 0$ gives
	\begin{equation}
		3d_i-\sum_{j=1}^{11}l_{i,j}=0. 
        \label{eq: relation of d, l_ij}
	\end{equation} 
Since $N_{Y_i} \cdot E_{i,j} = l_{i,j} > 0$ for all $1 \leq j \leq 11$, we have  $d_i > 0$.
For $N_{Y_i}$ to be base point free on $Y_i$, we must have  $	N_{Y_i}|_ {F_i}\simeq \cO_{F_i}$. Since
	\begin{equation}\nonumber
	\begin{aligned}
		N_{Y_i}|_ {F_i}
		=&f_i^*(\cO_{ \bP^ 2}(d_i))|_{F_i}\otimes\cO_{F_i}(-\sum_{j=1}^{11}l_{i,j}\widetilde{p_{i,j}})\\
		=&f_i^*(\cO_{ \bP^ 2}(d_i)|_{C}\otimes\cO_{C}(-\sum_{j=1}^{11}l_{i,j}p_{i,j}))\\
		=&f_i^*(\cO_{C}(\sum_{j=1}^{11}l_{i,j}(p_0-p_{i,j}))),	
	\end{aligned}
\end{equation}
it follows that
\begin{equation}
	\cO_{C}(\sum_{j=1}^{11}l_{i,j}(p_0-p_{i,j}))\simeq \cO_C.
    \label{eq: resrtiction condition}
\end{equation}

Since the orders $n_{i,1}, \dots, n_{i,10}$ of $p_{i,1}, \dots, p_{i,10}$ are fixed, we can choose a strictly increasing subsequence of $n_{i,11}$, the order of $p_{i,11}-p_0$,  such that $n_{i,11}$ is coprime to each  of $n_{i,1}, \dots, n_{i,10}$. 
Thus, we have $l_{i,11} \to +\infty$ as $i \to +\infty$ by (\ref{eq: resrtiction condition}). Therefore, (\ref{eq: relation of d, l_ij}) implies that $d_i \to +\infty$ as $i \to+ \infty$.
    Hence, we conclude that
	 $$N_i \cdot (K_{X_i} + M_i)= N_{Y_i}\cdot M_{Y_i} =11d_i-3\sum_{j=1}^{11}l_{i,j}=2d_i\to +\infty$$
 as $i\to +\infty$.
\end{proof}

\begin{rem}\label{rem: non-proj of X_i}

If we choose $p_{i,11}$ such that $\sum_{j=1}^{11}(p_{i,j}-p_0)\in \Pic^0(C)$ is non-torsion, then $A_i+2F_i$ is nef and big but not semiample in Lemma \ref{lem:bpf}. In this case, $X_i$  exists as a complex analytic space \cite{grauertUberModifikationenUnd1962}, or even as an algebraic space \cite{artinAlgebraizationFormalModuli1970}. However, $Y_i$ cannot be contracted to a projective surface $X_i$. Indeed, assume that $X_i$ is projective, then there exists a very ample divisor $N_i$ on $X_i$. By (\ref{eq: resrtiction condition}) in the proof of Proposition \ref{prop: M unbdd}, we have  $\cO_C(\hat{n}_il_{i,11}(p_0-p_{i,11}))\simeq\cO_C$, where $\hat{n}_i=\prod_{j=1}^{10} n_{i,j}$, which is a contradiction.
\end{rem}

The failure of boundedness of the underlying varieties is more subtle. We also need to take non-projective surfaces into consideration. 
The following lemma is also mentioned in \cite[Example 6]{kollarSeshadrisCriterionOpenness2022}, which shows that projectivity is not an open condition in the family of algebraic spaces  with singularities slightly worse than rational singularities.

\begin{lemma}\label{lem:analytic family}
There exists a family of algebraic spaces $h: \cX \to S$ such that the projective surfaces $X_i\in \cP_0$  correspond to a  countable dense subset $S^\tau$ of $S$.
\end{lemma}
\begin{proof}

Fix $C \subset \mathbb{P}^2$  a smooth elliptic curve  and distinct points
 $p_1, \dots, p_{10} \in C$. Consider the projection
 $$\bP^2\times S\to S\simeq C\setminus\{p_1,\ldots,p_{10}\}.$$
Let $g:\cY\to S$  be the family of surfaces given by blowing up the constant sections $p_1\times S,\ldots,p_{10}\times S$ and the diagonal section
$$\Gamma=\{(p_s,s)\subset \bP^2\times S\mid s\in S\}$$
in $\bP^2\times S$ where $p_s\in C$ is the point corresponding to $s\in S$  under the given isomorphism $S\simeq C\setminus\{p_1,\ldots,p_{10}\}$. 
Thus, for any $s\in S$, we have that $\mathcal Y_s$ is the blow up of $\bP^2$ along $p_1,\ldots,p_{10},p_s$.
Let $\cE$ be the exceptional divisor for $f:\cY\to \bP^2\times S$, $\cH$ be the pullback of $\cO_{\bP^2}(1)\times S$ via $f$, and $\cF$ be the strict transform of $C\times S$.

Let $\cI$ be the ideal sheaf  for $\cF$ in $\cY$, and 
$g_0:=g|_{\cF}:\cF\to S$. Then $\cI/\cI^2\simeq -\cF|_{\cF}$ is ample over $S$. Therefore, we have
\begin{enumerate}
        \item For every coherent sheaf $\cG$ on $\cF$,         $R^1g_{0*}(\cG\otimes(\cI/\cI^2)^{\otimes n})=0$  for sufficiently large $n$, by relative Serre vanishing.
        \item  $R^1g_{0*}(\cI/\cI^2)^{\otimes n}=0$ for every $n\geq 1$, by relative Kodaira vanishing.
\end{enumerate}
By Theorem \ref{Thm:Arin contraction}, we have a contraction $\pi:\cY\to \cX$ to an algebraic space $\cX$ such that the restriction of $\pi$ to $\cY\setminus \cF$ is an isomorphism and it contracts $\cF$ to a section  of $\cX\to S$. 
Then the induced family $h:\cX\to S$ is a family of algebraic spaces.
Let $$S^\tau:=\{s\in S\mid p_1+\ldots+p_{10}+p_s \text{ is torsion on $C$}\}, \quad \text{ and }\quad S^*:=S\setminus S^\tau.$$
Here $p_1+\ldots+p_{10}+p_s$ is the sum in $C$ according to the group law on $C$. 
Then by Lemma \ref{lem:bpf},
$\cM_{\cY}:=11\cH-3\cE$ is nef and big over $S$. Moreover,
$(\cM_{\cY})_s$ is semiample for $s\in S^\tau$, which contracts $\cY_s$ to a projective surface $\cX_s$,  while $(\cM_{\cY})_s$ is not semiample  for $s\in S^*$. Thus $\cX_s$ is projective if and only if $s\in S^\tau$.
\end{proof}

\begin{thm}\label{thm:P_0 unbdd}
    The set of   underlying varieties   appearing in $\cP_0$ is not bounded.
\end{thm}
\begin{proof}

Let $g:\cY\xrightarrow[]{\pi} \cX \xrightarrow[]{h} S$ be the two families obtained in Lemma \ref{lem:analytic family}, where $\cY$ is projective over $S$ but $\cX$ is not projective over $S$ ($\cX_s$ is projective if and only if $s \in S^\tau$). Suppose that the set $\{\cX_s \mid s \in S^\tau \}$ is bounded. Then, there exists a projective family $h':\cX' \to T$ such that for every $s \in S^\tau$, there exists some $t = t(s) \in T$ with $\cX_s \cong \cX'_t$. We may assume that $\cX' \to T$ is a locally closed subset of the corresponding Hilbert scheme. After replacing $T$ with the closure of  $S^\tau$ in $T$, we may assume that $S^\tau$ is dense in $T$.


After  replacing $T$ with a dense open subset and $S^\tau$ with the corresponding subset, we can take a common minimal resolution of $\cX'_t$, and it is just the blowing up of a section of $\cX'\to T$, and so we get a morphism $\pi':\cY' \to \cX'$ with an irreducible exceptional divisor $\cF'$. Moreover,
for every $s \in S^\tau$, there exists some $t = t(s) \in T$ 
such that $\cY'_t \supset \cF'_t \cong \cF_s \subset \cY_s$, where $\cF_s$ is the strict transform of the elliptic curve $C\subset \bP^2$. 
     \[
\begin{tikzcd}
\mathcal{Y} \arrow[rrrrr, bend left=12]  \arrow[d, "\pi"]&&\cY_s \arrow[d] \arrow[r, phantom, "\simeq"]&\cY'_t\arrow[d] &\cY'\arrow[r] \arrow[d, "\pi'"]&\cU\arrow[dd]\\
 \mathcal{X} \arrow[d, "h"] & & \cX_s\arrow[d]  \arrow[r, phantom, "\simeq"]&\cX'_t \arrow[d]  &\cX' \arrow[d, "h'"]  \\
S \arrow[rrrrr, bend right=14] & S^\tau \arrow[l, phantom, "\mathrel{\supset}"]    & s \arrow[r] \arrow[l, phantom, "\ni"]& t=t(s) \arrow[r, phantom, "\in"]& T \arrow[r]&H
\end{tikzcd}
\]

Fix a relatively very ample divisor $\cA'$ for $g':\cY' \to T$ such that $H^1(\cY_t',\cO_{\cY'_t}(\cA'_t))=0$ for $t\in T$. 
Let $\cU \to H$ be the corresponding universal family over the Hilbert scheme $H$. Then there is a morphism $T \to H$ such that $\cY' = \cU_T:=\cU\times _H T$. For every $s \in S^\tau$, we have a corresponding very ample divisor $\cA_s$ of fixed degree on $\cY_s$. Since $H^i(\cY_s,\mathcal O_{\mathcal Y_s})=0$ for $i=1,2$ and $s\in S$, then $H^1(\cY_s,\mathcal O ^*_{\mathcal Y_s})\cong H^2(\mathcal Y_s,\mathbb Z)$ for $s\in S$.
Thus, after replacing $S$ with an \'etale cover, there exists a line bundle $\cL$ on $\cY$ such that $\cL_s \cong \mathcal{O}_{\cY_s}(\cA_s)$ for every $s \in S^\tau$. 
Since $H^1(\cY_s, \cO_{\cY_s}(\cA_s)) = 0$ for $s \in S^\tau$, by \cite[Proposition 7.7]{nugentHigherDimensionalBrillNoether2024}, after replacing $S$ with an open neighborhood of some $s \in S^\tau$, we may assume that $\cL_s$ is very ample with a fixed degree for all $s \in S$.
This yields a morphism $S \to H$ such that $\cY = \cU_S:=\cU\times _H S$.

 Note that the closures of the images of $S$ and $T$ in $H$ coincide, leading to a common cover $R \to S$ and $R \to T$ such that $\cY_R:=\cY\times_S R$ coincides with $\cY'_R:=\cY'\times_T R$. Consequently, for every $s \in S$, there exists some $t \in T$ with an isomorphism $\cY_s \cong \cY'_t$.

Let $\cN'$ be the pullback of a relative very ample divisor on $\cX'$ to $\cY'$. Then for every $s \in S$, there exists some $t \in T$ such that $\cN'_t \subset \cY'_t$ gives a base point free divisor $\cN_s \subset \cY_s$. However, $|\cN_s|$ contracts $\cF_s$, where $\cF_s$ is the strict transform of the elliptic curve $C\subset \bP^2$, which implies that $\cX_s$ is projective, leading to a contradiction.
\end{proof}


	\begin{rem}

    We say that $X$ is  a \textit{polarized  Calabi-Yau variety} if $X$ is slc, $K_X\qlin 0$, and there exists  an ample $\bQ$-Cartier Weil divisor $M$ on $X$.

    Let $d\in \bN$ and $v\in\bQ^{>0}$. Fix $\dim X=d$ 
and $\vol(M)=v$.  
In this subsection, we construct a set of polarized  Calabi-Yau varieties with $d=2$,  and $v=22$.
Our example shows that to bound $X$, we cannot remove the klt condition in \cite[Corollary 1.6]{birkarGeometryPolarisedVarieties2023}, or the condition that $M\geq 0$ and $(X,tM)$ is slc for some $0<t\ll 1$ in \cite[Corollary 1.8]{birkarGeometryPolarisedVarieties2023}.
	\end{rem}

\begin{exam}
    We say that $(W,D)$ is  a \textit{boundary polarized  Calabi-Yau pair} \cite{ascherModuliBoundaryPolarized2023,blumGoodModuliSpaces2024} if $(W,D)$ is slc, $K_W+D\qlin 0$, and $D$ is an ample $\bQ$-Cartier $\bQ$-divisor on $W$.
     Note that $-K_W$ is ample and so $W$ is an slc Fano variety.

    Let $d\in \bN$ and $v\in\bQ^{>0}$. Fix $\dim W=d$ 
and $\vol(-K_W)=v$. For $(X_i,M_i)\in \cP_0$ constructed in this subsection, let 
$$W_i=C_p(X_i, M_i) := \text{Proj} \bigoplus_{n=0}^\infty \bigoplus_{m=0}^\infty H^0(X_i, \mathcal{O}_{X_i}(mM_i)) \cdot t^n$$
be the projective cone over $X_i$, and $X_i^\infty\simeq X_i$ be the divisor at infinity. Then by  \cite[Proposition 2.20.3]{ascherModuliBoundaryPolarized2023}, $(W_i,X_i^\infty)$ is a boundary polarized Calabi-Yau pair.
Let $\gamma_i:V_i\to W_i$ be the blow up of the vertex with exceptional divisor $X_i^-$, set $X_i^+=\gamma_i^*X_i^\infty$,
and let $\beta_i:V_i\to X_i$ be the associated $\bP^1$-bundle.
Let $\tau_i:U_i\to V_i$ be the resolution of singularities with exceptional divisor $R_{F_i}\simeq \alpha_i^*F_i$ where  $\alpha_i:U_i\to Y_i$ is the associated $\bP^1$-bundle, and set $Y_i^\pm=\tau_i^*X_i^\pm$. 
We have the following commutative diagram.

\[
\begin{tikzcd}
(U_i,R_{F_i}+Y_i^-+Y_i^+) \arrow[r, "\tau_i"] \arrow[d,"\alpha_i"]&(V_i,X_i^-+X_i^+)\arrow[r,"\gamma_i"]\arrow[d,"\beta_i"]&(W_i,X_i^\infty)\\
 (Y_i,F_i+M_{Y_i})\arrow[r, "\pi_i"] & (X_i,M_i) 
\end{tikzcd}
\]

We claim that the Cartier index of $-K_{W_i}$ is unbounded, and hence $(W_i,X_i^\infty)$ is unbounded by the same argument as in \cite[\S5.3, p.20]{birkarVariationsGeneralisedPairs2022}.
Indeed, if the Cartier index of $-K_{W_i}$ is bounded, then by \cite[Corollary 1.6]{jiangBoundednessLogPluricanonical2021}, there exsits $q\in \bN$ such that $q(K_{W_i}+X_i^\infty)\sim 0$ for for all $i$.
Hence, the Cartier index of $X_i^\infty$ is also bounded. Since $X_i^\infty-(-K_{W_i})$ is ample, by the effective base point free theorem \cite[Theorem 1.1]{fujinoEffectiveBasePoint2009} and the very ampleness lemma \cite[Lemma 7.1]{fujinoEffectiveBasepointfreeTheorem2017} for lc pairs, we can find $l\in \bN$ such that $lX_i^\infty$ is very ample for all $i$.
Therefore, $lX_i^+=\gamma_i^*(lX_i^\infty)$ is big and base point free on $V_i$ for all $i$. 
However, since $Y_i^+|_{Y_i^+}\sim M_{Y_i}$, it follows that 
$X_i^+|_{X_i^+}\sim M_i$, which implies that 
the Cartier index of $X_i^+$ is not bounded, leading to a contradiction.

Note that $(U_i,\alpha_i^*F_i+Y_i^-+Y_i^+)$ is a dlt modification of $(W_i,X_i^\infty)$ with minimal non-klt centers $F_i^-:=R_{F_i}\cap Y_i^-\simeq F_i$ and $F_i^+:=R_{F_i}\cap Y_i^+\simeq F_i$. 
Therefore, we obtain an unbounded set of boundary polarized Calabi-Yau 3-fold pairs $(W_i, X_i^\infty)$ such that $\vol(-K_{W_i})=(X_i^\infty)^3=M_i^2=22$ and $\text{reg}(W_i, X_i^\infty)=1$ (see \cite[Deﬁnition 8.5]{ascherModuliBoundaryPolarized2023}). This case is not treated in \cite[\S8]{ascherModuliBoundaryPolarized2023}.
\end{exam}

\subsection{Minimal surfaces of Kodaira dimension one}

	In this subsection, we construct an unbounded set of generalised pairs $(U_i, B_{U_i} + M_{U_i}) \in\cP_1\subset \cF_{glc}(2, \Phi, v)$ for $\Phi = \{0,1
\}$ and $v = 68$, such that $K_{U_i}$ is semiample with $\kappa(U_i)=1$, $B_{U_i} = 0$, and $M_{U_i}$ is an ample Weil divisor with  unbounded Cartier index.

In the example of \S\ref{S:K=0}, let $D \subset \bP^2$ be another smooth elliptic curve, and assume that $p_{i,1}, \ldots, p_{i,9} \in C \cap D$. We first blow up $\bP^2$ at $p_{i,1}, \ldots, p_{i,9}$, giving a pencil spanned by the strict transforms of $C$ and $D$. Then, we further blow up $p_{i,10}$ and $p_{i,11}$ on the strict transform of $C$,  we obtain a surface $Y_i$
with a morphsim $Y_i\to \bP^1$.
Let $M_{Y_i}$ be the big and semiample Cartier divisor on $Y_i$ that defines a contraction $\pi_i : Y_i \to X_i$ which contracts the strict transform $F_i$ of $C$ on $Y_i$. Let $M_i = \pi_{i*} M_{Y_i}$, and let $X_i \to \bP^1$ be the induced morphism.

Let $G_{i,1}$ and $G_{i,2}$ be two general fibers of $Y_i \to \bP^1$ that do not contain $F_i$. Let $\nu_i:V_i\to Y_i$ be the double cover of $Y_i$ branched over $G_{i,1}$ and $G_{i,2}$, then $M_{V_i}: = \nu_i^* M_{Y_i}$ is also a big and semiample Cartier divisor.
Since $Y_i$ is smooth and $G_{i,1}$, $G_{i,2}$ are disjoint smooth fibers, it follows that $V_i$ is also smooth.
Let $\mu_i:U_i\to X_i$ be the double cover of $X_i$ branched over $\pi_{i*} G_{i,1}$ and $\pi_{i*} G_{i,2}$, then $M_{U_i}: = \mu_i^* M_i$ is also ample with unbounded Cartier index.
We have the following commutative diagram, where $\tau_i : \bP^1 \to \bP^1$ is a double cover branched over two points.

\[
\begin{tikzcd}
(V_i,\nu_i^*F_i+M_{V_i}) \arrow[r, "\nu_i"] \arrow[d, "\gamma_i"] & (Y_i,F_i+M_{Y_i}) \arrow[d, "\pi_i"] \\
(U_i,M_{U_i}) \arrow[r, "\mu_i"] \arrow[d] & (X_i,M_i) \arrow[d] \\
\bP^1 \arrow[r, "\tau_i"] & \bP^1
\end{tikzcd}
\]
Since $(Y_i,F_i+\frac{1}{2}(G_{i,1}+G_{i,2})+M_{Y_i})$ is generalised lc, and 
$$K_{V_i}+\nu_i^*F_i+M_{V_i}=\nu_i^*(K_{Y_i}+F_i+\frac{1}{2}(G_{i,1}+ G_{i,2})+M_{Y_i}),$$
by \cite[Proposition 5.20]{kollarBirationalGeometryAlgebraic1998}, $(V_i,\nu_i^*F_i+M_{V_i})$ is also generalised lc.
Since
$$\gamma_i^*K_{U_i}=K_{V_i}+\nu_i^*F_i \quad \text{and} \quad \gamma_i^*M_{U_i}=M_{V_i},$$
it follows that $(U_i,M_{U_i})$ is a generalised lc surface with data $V_i\to U_i$ and $M_{V_i}$.
Moreover,
$$K_{U_i}=\mu_i^*(K_{X_i}+\frac{1}{2}(\pi_{i*} G_{i,1}+\pi_{i*} G_{i,2}))$$
is semiample with $\kappa(K_{U_i})=1$.

Recall that $M_{Y_i}\sim 11H_i-3E_i$, and $G_{i,1}\sim G_{i,2}\sim 3H_i-(E_{i,1}+\cdots+E_{i,9})$.
We have 
$$M_i\cdot \pi_{i*}G_{i,k}=M_{Y_i}\cdot G_{i,k}=33-3\times 9=6$$
for $k=1,2$.
 Now, $K_{U_i}+M_{U_i}$ is ample and 
\begin{align*}
\vol(K_{U_i} + M_{U_i}) &= \mu_i^*(K_{X_i} + \frac{1}{2}(\pi_{i*} G_{i,1} + \pi_{i*} G_{i,2}) + M_i) ^2\\
&= 2 (K_{X_i} + \frac{1}{2}(\pi_{i*} G_{i,1} + \pi_{i*} G_{i,2}) + M_i) ^2\\
&= 2 (2M_i \cdot \pi_{i*} G_{i,2} + M_i^2) \\
&= 2 (12 + 22) \\
&= 68.
\end{align*}
Thus 
	 $$(U_i,M_{U_i})\in\cP_1\subset \cF_{glc}(2,\Phi,v)$$
	 for $\Phi=\{0,1\}$ and $v=68$. 
\begin{rem}

If we choose $p_{i,11}$ such that $\sum_{j=1}^{11}(p_{i,j}-p_0)\in \Pic^0(C)$ is non-torsion, then $U_i$ is not projective. Indeed, assume that $U_i$ is projective, then there exists an ample divisor $N_{U_i}$ on $U_i$. Since the double cover $\mu_i: U_i \to X_i$ is always a Galois cover with group $G = \mathbb{Z}_2$, define $N_{U_i}^G := \frac{1}{|G|} \sum_{g \in G} g^* N_{U_i}$. Then $N_{U_i}^G$ is $G$-invariant, so there exists an ample divisor $N_i$ on $X_i$ such that $N_{U_i}^G = \mu_i^* N_i$. This implies that $X_i$ is projective, which is a contradiction.
 \end{rem}   
\begin{thm}\label{thm:P_1 unbdd}
    The set of   underlying varieties   appearing in $\cP_1$ is not bounded.
\end{thm}
\begin{proof}
Suppose that $U_i$ belongs to a bounded family. The induced fibration $U_i\to \mathbb P^1$ is just the Iitaka fibration and hence it also belongs to a bounded family.
Let $\alpha_i:U_i\to U_i$ be the involution that fixes the ramification divisor of $\mu_i:U_i\to X_i$. From the construction in this subsection, we know that the involution $\alpha_i$ is determined by an involution $\beta:\bP^1\to \bP^1$ that fixes two points of $\bP^1$, and hence $\alpha_i$ is bounded.

It then follows that $X_i=U_i/<\alpha _i >$ is also bounded, which contradicts Theorem \ref{thm:P_0 unbdd}.
\end{proof}

\begin{rem}

    We say that $U$ is  a \textit{stable minimal model} if $U$ is slc, $K_U$ is semiample defining a contraction $f:U\to Z$, and there exists  a $\bQ$-Cartier Weil divisor $M_U$ on $U$ which is ample over $Z$. 
    In particular, $F$ is a polarized Calabi-Yau variety with polarization $M_U|_F$, where $F$ is the general fiber of $f:U\to Z$.

    Let $d\in \bN$, $\sigma\in\bQ[t]$ be a polynomial, and $u\in\bQ^{>0}$.
    Fix $\dim U=d$, $(K_U+tM_U)^d=\sigma(t)$, 
 and $\vol(M_U|_F)=u$. 
 In  this subsection, we have constructed a set of   stable minimal models with $d=2$, $\sigma(t)=24t+44t^2$, and $u=12$.
 Our example  shows that to bound $U$, we cannot remove the klt condition in \cite[Theorem 1.4]{jiangBoundednessModuliTraditional2023}, or the condition that $M_U\geq 0$ and $(U,tM_U)$ is slc for some $0<t\ll 1$ in \cite[Theorem 1.12]{birkarModuliAlgebraicVarieties2022}.
 Moreover, in our example, both the general fiber $F$ and the base $Z$ are bounded, while previously known examples of the failure of boundedness of stable minimal models with fixed $(d,\sigma,u)$ arise from the failure of boundedness of the general fibers.
\end{rem}

\subsection{Minimal surfaces of general type}

	In this subsection, we construct an unbounded set of generalised pairs $(X_i, B_i + M_i) \in\cP_2\subset \cF_{glc}(2, \Phi, v)$ 
such that $K_{X_i}$ is nef and big but not ample, $B_i = 0$, and $M_i$ is an ample divisor with unbounded Cartier index.

    Let $C \subset \mathbb{P}^2$ be a smooth cubic, and take $r$ general lines $L_j \subset \mathbb{P}^2$.
Let $h: W \to \bP^2$ be the blow up of all singular points of $C + \sum L_j$.  
Denote the strict transform of $C + \sum L_j$ by $C' + \sum L_j'$.  
Define $H := h^*\cO_{\bP^2}(1)$.  
For $1 \leq j \leq r$, $1 \leq l \leq 3$, let $E_{jl}$ be the exceptional divisor over $C \cap L_j$.  
Similarly, for $1 \leq j < k \leq r$, let $F_{jk}$ be the exceptional divisor over $L_j \cap L_k$.
    Let $\tau: W \to Z$ be the contraction of  $C' + \sum L_j'$. Then $Z$ has one simple elliptic singularity (the image of $C$) and $r$ quotient singularities (the image of the $L_j$). Moreover,  $K_Z$ is ample if $r \geq 4$.
Since our computation does not coincide with \cite[Example 5]{kollarDeformationsVarietiesGeneral2021} (where it is stated that $K_Z$ is ample if $r \geq 6$), we provide the details of the computation.
\begin{lemma}
      If  $r \geq 4$, then $K_Z$ is ample.
\end{lemma}
\begin{proof}

Note that we have the following relations:
\begin{equation} \label{eq:K_W}
K_W = h^*K_{\bP^2} + \sum_{1\leq j\leq r, 1\leq l\leq3} E_{jl} + \sum_{1\leq j<k\leq r}F_{jk},
\end{equation}  

\begin{equation} \label{eq:C_prime}
C' = 3H - \sum_{1\leq j\leq r, 1\leq l\leq3} E_{jl},
\end{equation}  

\begin{equation} \label{eq:L_j_prime}
L_j' = H - \sum_{1\leq l\leq 3}E_{jl} - \sum_{1\leq k\leq r, k\neq j}F_{jk}.
\end{equation}  
From (\ref{eq:C_prime}) and (\ref{eq:L_j_prime}), we deduce that
\begin{equation} \label{eq:F_jk}
2\sum_{1\leq j<k\leq r}F_{jk} = \sum_{1\leq j\leq r} (H - \sum_{1\leq l\leq 3}E_{jl} - L_j') = rH - (3H - C') - \sum_{1\leq j\leq r}L_j'.
\end{equation}  
Then by (\ref{eq:K_W}), (\ref{eq:C_prime}), and (\ref{eq:F_jk}), we have
\begin{align*}
    K_W = \frac{r-3}{2}H - \frac{1}{2}C' - \frac{1}{2} \sum_{1\leq j\leq r}L_j'.
\end{align*} 
Since $L_j'$ is a rational curve with $(L'_j)^2=1-3-(r-1)=-1-r$, the log discrepancy is given by $a(L_j',Z,0)=\frac{2}{r+1}$.
Thus, we have
\begin{align}
\tau^*K_Z &= K_W + C' + \sum_{1\leq j\leq r} (1 - \frac{2}{r+1}) L_j'   \notag \\
           &= \frac{r-3}{2}H + \frac{1}{2}C' + \frac{r-3}{2(r+1)} \sum_{1\leq j\leq r}L_j'.  \notag 
\end{align}
 If $r\geq 4$, then for any irreducible curve $\widetilde{C}$ such that $\widetilde{C} \neq C'$ and $\widetilde{C} \neq L_j'$ for all $j$,  we have $\tau^*K_Z\cdot \widetilde{C}\geq \frac{r-3}{2}>0$. Since $\tau^*K_Z\cdot C'=\tau^*K_Z\cdot L_j'=0$, it follows that if $r\geq 4$, then
\begin{align}
(\tau^*K_Z)^2 &= ( K_W + C' + \sum_{1\leq j\leq r} ( 1 - \frac{2}{r+1} ) L_j' ) \cdot K_W \notag \\
&= K_W^2 + -C'^2 + \sum_{1\leq j\leq r} ( 1 - \frac{2}{r+1} )(-2 - L_j'^2) \notag \\
&= (9 - 3r - \frac{r(r-1)}{2}) + (3r - 9) + r(1 - \frac{2}{r+1})(-2 + r + 1) \notag \\
&= \frac{r(r-1)(r-3)}{2(r+1)}>0. \notag
\end{align}
Thus, by the Nakai-Moishezon criterion, $K_Z$ is ample if $r\geq 4$.
\end{proof}

Pick a point $p_i\in C'\setminus(C'\cap\sum L_j')\subset W$ such that $p_i$ is torsion with order $n_i$. 
Let $g_i:Y_i \to W$ be the blow up of $W$ at $p_i$ with exceptional divisor $E_i$, 
and denote by $C_i' + \sum L_{ij}'$ the strict transform of $C' + \sum L_j'$ on $Y_i$. 
Then $g_i^*C'=C_i'+E_i$.
Define $H_i := g_i^*H$, $E_{ijl} := g_i^*E_{jl}$ for $1 \leq j \leq r$, $1 \leq l \leq 3$,  
and $F_{ijk} := g_i^*F_{jk}$ for $1 \leq j < k \leq r$.
    Let $\pi_i: Y_i \to X_i$ be the contraction of  $C_i' + \sum L_{ij}'$.
    Define $E_{X_i} := \pi_{i*} E_i$. Then $f_i: X_i \to Z$ is a contraction with a single exceptional divisor $E_{X_i}$.
    We have the following commutative diagram.
\[
\begin{tikzcd}
(Y_i\supset H_i,C_i',L_{ij}',E_i,E_{ijl},F_{ijk}) \arrow[r, "g_i"] \arrow[d,"\pi_i"]&(W\supset H,C',L_j',E_{jl},F_{jk}) \arrow[r,"h"]\arrow[d,"\tau"]&\bP^2\\
 (X_i\supset E_{X_i})\arrow[r, "f_i"] & Z 
\end{tikzcd}
\]
\begin{lemma}
    $K_{X_i}$ is nef and big but not ample.
\end{lemma}
\begin{proof}
  We may write
$$K_{X_i} = f_i^*K_Z + mE_{X_i}$$
for some $m \in \bQ$. Then, together with the following relations:
    $$K_W + C' + \sum_{1\leq j\leq r} (1 - \frac{2}{r+1}) L_j' =\tau^*K_Z,$$
    $$K_{Y_i}=g_i^*K_W+E_i,$$
    \begin{equation}
K_{Y_i} + C_i' + \sum_{1\leq j\leq r} (1 - \frac{2}{r+1}) L_{ij}' = \pi_i^* K_{X_i}, \label{eq:KY_C_L}
\end{equation}
    $$g_i^*C'=E_i+C_i',$$
    $$g_i^*L_j'=L_{ij}',$$
    we deduce that $m = 0$, hence $K_{X_i} = f_i^*K_Z$ is nef and big. Moreover, since
$K_{X_i} \cdot E_{X_i} = 0$,
it follows that $K_{X_i}$ is not ample.
\end{proof}

We aim to construct an ample $\bQ$-Cartier divisor $M_i$ on $X_i$ such that its Cartier index is unbounded. The argument is similar to  Lemma \ref{lem:bpf}.

Fix $r\geq 4$ and   let $d\geq 2(3r+\binom{r}{2}+2)$ be a fixed positive integer. Then, by Theorem \ref{thm:arbitrary blow up P2},
 $$A_i:=dH_i-\sum E_{ijl}-\sum F_{ijk}-E_i$$ is very ample on $Y_i$.
Let  $(a, b)$ be the unique solution of the system of linear equations
$$
\left\{
\begin{aligned}
(A_i + a C_i' + b \sum L_{ij}') \cdot C_i' = 0, \\
(A_i + a C_i' + b \sum L_{ij}') \cdot \sum L_{ij}' &= 0,
\end{aligned}
\right.
$$
and let $q\in \bN$ such that $M_{Y_i}:=q(A_i + a C_i' + b \sum L_{ij}')$ is an integral divisor. Since $a=-A_i\cdot C_i'/(C_i')^2>0$ and $b=-A_i\cdot L'_{ij}/(L'_{ij})^2>0$ for any $j$, it follows that 
 $M_{Y_i} \cdot \widetilde{C}_i > 0$ for any irreducible curve $\widetilde{C}_i$ on $Y_i$ with $\widetilde{C}_i \neq C_i'$ and $\widetilde{C}_i \neq L_{ij}'$ for all $j$, and hence $M_{Y_i}$ is nef. Moreover, since $M_{Y_i}$ is the sum of an ample divisor and  effective divisors, it follows that $M_{Y_i}$ is big.  Let $m_i\in \bN$ such that $\cO_{Y_i}(m_iM_{Y_i})\otimes\cO_{C_i'}\simeq \cO_{C_i'}$, where $m_i$ depends on the torsion order of $p_i$.

\begin{lemma}
   The linear system $|m_iM_{Y_i}|$ is base point free. 
\end{lemma}
\begin{proof}

Let $D_i$ be either the curve $C_i'$ or $L_{ij}'$, since $D_i^2<0$ and $g(D_i)\leq 1$, we have $H^1(D_i, \cI_{D_i}^t/\cI_{D_i}^{t+1} )=0$ for $t\geq 1$, where $\cI_{D_i}:=\cO_{Y_i}(-D_i)$ is the ideal sheaf of $D_i$ in $Y_i$.
Thus, the cohomology sequence induced by the truncated exponential exact sequences
    $$0\to \cI_{D_i}^t/\cI_{D_i}^{t+1} \to  \cO_{(t+1)D_i}^*\to \cO_{tD_i}^*\to 0$$
    yields the fact that the restriction maps
	 $$
\Pic((t+1)D_i) \to \Pic(tD_i)$$
are isomorphisms for all $t\geq 1$.
    Define $G_i:=m_iq(a C_i' + b \sum L_{ij}')$. 
    Since 
$$\cO_{Y_i}(m_iM_{Y_i})\otimes\cO_{\Supp  (G_i)}\simeq \cO_{\Supp  (G_i)},$$
 we have
$$\cO_{Y_i}(m_iM_{Y_i})\otimes\cO_{G_i}\simeq \cO_{G_i}.$$
Thus, we obtain the following short exact sequence
$$
0 \to \mathcal{O}_{Y_i}(m_iM_{Y_i} - G_i) \to \mathcal{O}_{Y_i}(m_iM_{Y_i}) \to \mathcal{O}_{G_i} \to 0.
$$
This gives the long exact sequence 
$$
H^0(Y_i, \cO_{Y_i}(m_iM_{Y_i})) \to H^0(G_i, \cO_{G_i}) \to H^1(Y_i, \cO_{Y_i}(m_iqA_i)).
$$
Since $m_iqA_i - K_{Y_i}$ is very ample by Theorem \ref{thm:arbitrary blow up P2}, it follows that  $H^1(Y_i, \cO_{Y_i}(m_iqA_i)) = 0$ by Kodaira's vanishing.
Therefore, $H^0(Y_i, \cO_{Y_i}(m_iM_{Y_i})) \to H^0(G_i, \cO_{G_i})$ is surjective, and hence the base points of $|m_iM_{Y_i}|$ are not contained in $\Supp (G_i)$. Moreover, since $A_i$ is very ample, any base point of $|m_iM_{Y_i}|$ must be contained in $C_i'$ or $L_{ij}'$.    Thus, $|m_iM_{Y_i}|$ is base point free.  
\end{proof}

Now, let $M_i := \pi_{i*}M_{Y_i}$ be the corresponding ample divisor on $X_i$, where the  Cartier index of $M_i$ is unbounded. Then $M_{Y_i}= \pi^*M_i$, and it follows from
  by (\ref{eq:KY_C_L}) that
 $(X_i, M_i)$ is a generalised lc surface with data $Y_i \to X_i$ and $M_{Y_i}$. 
Moreover,  $K_{X_i}+M_i$ is ample and 
$\vol(K_{X_i}+M_i)$ is fixed.
Thus,
$$(X_i,M_i)\in\cP_{2}\subset \cF_{glc}(2,\Phi,v)$$
	 for  $\Phi=\{0,1\}$ and some fixed $v$. 

\begin{thm}\label{thm:P_2 unbdd}
    The set of   underlying varieties   appearing in $\cP_2$ is not bounded.
\end{thm}
\begin{proof}
    If we choose a non-torsion point $p_i \in C' \setminus (C' \cap \sum L_j') \subset W$, 
then  by the same argument as in Remark \ref{rem: non-proj of X_i},  $X_i$ is non-projective. Moreover, we have
$$H^1(Y_i,\cO_{Y_i})=H^2(Y_i,\cO_{Y_i})=0$$ 
for every choice of $p_i \in C' \setminus (C' \cap \sum L_j') \subset W$.
Therefore, we can use the same argument as in Theorem \ref{thm:P_0 unbdd} to show that $\cP_{2}$ is not bounded.
\end{proof}

\begin{rem}
Given a variety $Z$ and a birational contraction $f:X\to Z$, we can write
$$K_X+B=f^*K_Z$$
for some uniquely determined $B$. We say $(X,B)$ is a \textit{crepant model} of $Z$ if $B\geq 0$.
    By the birational case of \cite[Theorem 1.2]{birkarBoundednessFanoType2024}, we know that 
    if $Z$ is $\epsilon$-lc for some fixed $\epsilon\in\bQ^{>0}$ and it belongs to a bounded family, 
    then the underlying varieties of all the crepant models of $Z$ form a bounded family. Our example in this subsection shows that
    the $\epsilon$-lc condition is necessary.
\end{rem}

\subsection{Weak Fano surfaces}
	In this subsection, we construct an unbounded set of generalised pairs $(X_i, B_i + M_i) \in\cP_{-\infty}\subset \cF_{glc}(2, \Phi, v)$
     for $\Phi=\{0,\frac{1}{8}\}$ and $v=\frac{45}{16}$, 
such that $-K_{X_i}$ is nef and big but not ample, $B_i = 0$, and $M_i$ is an ample $\bQ$-divisor with  unbounded Cartier index. Our example is inspired by \cite[2.19]{kollarLecturesResolutionSingularities2007}.

Let $C\subset\bP^2$ be a smooth elliptic curve,  
and $Z$ be the projective cone  over $C$. Then $Z$ is a Fano surface with $\rho(Z) = 1$, and $K_Z + C^\infty \qlin 0$, where $C^\infty$ is the section of $Z$ at infinity.
Let $\tau:W\to Z$ be the blow up of the vertex with exceptional divisor $C^-$, and $W\to C$ be the associated $\bP^1$-bundle.
Fix $p_0\in C^-$, then we can embed $C^-$ into $\mathbb{P}^2$ such that $p_0$ serves as the identity element of the group structure on $C^-$.
Let $C^+$ be the positive section of $W\to C$, and $F$ be a fiber of $W\to C$.
By \cite[Proposition \uppercase\expandafter{\romannumeral5}.2.3]{hartshorneAlgebraicGeometry1977}, we have $\text{Pic}(W)=\bZ C^-\oplus \bZ F$, and $C^+=C^-+3F$.
Moreover, We have 
$$
-(C^-)^2 = (C^+)^2 = 3, \quad C^+ \cdot F = C^- \cdot F = 1, \quad \text{and} \quad F^2 = 0.
$$

Pick a point $p_i\in C^-\subset W$ such that  the divisor $p_i-p_0\in \Pic^0(C^-)$ has order $n_i$. 
Let $g_i:Y_i\to W$ be the  blow up of $W$ at $p_i$ with exceptional divisor $E_i$. 
Then $Y_i$ is bounded by \cite[Lemma 3.8]{alexeevBoundednessLogSurfaces1994}.
Denote $C_i^\pm := g_i^*C^\pm$ and $F_i := g_i^*F$ on $Y_i$. Then we have $C_i^- = E_i + G_i^-$, where $G_i^-$ is the strict transform of $C^-$ on $Y_i$. 
Since 
$$-3=(C_i^-)^2=(E_i + G_i^-)^2=-1+2+(G_i^-)^2,$$ it follows that 
$(G_i^-)^2=-4$.
By \cite[Proposition \uppercase\expandafter{\romannumeral5}.3.2]{hartshorneAlgebraicGeometry1977}, we have
$$\Pic(Y_i)=\bZ C_i^-\oplus \bZ F_i \oplus \bZ E_i.$$
Let $\pi_i: Y_i \to X_i$ be the contraction that contracts $G_i^-$. Define $E_{X_i} := \pi_{i*} E_i$. Then $f_i: X_i \to Z$ is a contraction with a single exceptional divisor $E_{X_i}$.
Write $\pi_i^*E_{X_i} = E_i + nG_i^-$. From $\pi_i^*E_{X_i} \cdot G_i^- = 0$, we deduce that $1 + n(-4) = 0$, hence $n = \frac{1}{4}$.
Therefore, 
$$(E_{X_i})^2=(E_i + \frac{1}{4}G_i^-)^2=(-1+\frac{1}{2}+\frac{1}{16}(-4))=-\frac{3}{4}.$$
Let  $C_{X_i}^+ := f_i^* C^\infty$.
We have the following commutative diagram.
\[
\begin{tikzcd}
(Y_i\supset C_i^+, C_i^-,F_i,E_i,G_i^-) \arrow[r, "g_i"] \arrow[d,"\pi_i"]&(W\supset C^+,C^-,F) \arrow[r]\arrow[d,"\tau"]&C\\
 (X_i\supset C_{X_i}^+,E_{X_i})\arrow[r, "f_i"] & (Z\supset C^\infty) 
\end{tikzcd}
\]

\begin{lemma}
  $-K_{X_i}$ is nef and big but not ample.  
\end{lemma}
\begin{proof}
Assume that
$$K_{X_i} = f_i^*K_Z + mE_{X_i}$$
for some $m \in \bQ$. Then, together with the following relations:
    $$K_W+C^-=\tau^*K_Z,$$
    $$K_{Y_i}=g_i^*K_W+E_i,$$
    $$K_{Y_i}+G_i^-=\pi_i^*K_{X_i},$$
    $$g_i^*C^-=E_i+G_i^-,$$
    we deduce that $m = 0$, hence $-K_{X_i} = -f_i^*K_Z$ is nef and big. Moreover, since
$-K_{X_i} \cdot E_{X_i} = 0$,
it follows that $-K_{X_i}$ is not ample.
\end{proof}

\begin{lemma}
     $N_i:=C_{X_i}^+-\frac{1}{2} E_{X_i}$ is ample on $X_i$.
\end{lemma}
\begin{proof}

Let  
$$N_{Y_i} := \pi_i^*N_i =( C_i^-+3F_i) - \frac{1}{2} ( E_i + \frac{1}{4} (C_i^--E_i) ) = \frac{7}{8} C_i^- + 3F_i - \frac{3}{8} E_i.$$  
Note that  
$$N_i^2 =(N_{Y_i})^2= (C_{X_i}^+)^2 + \frac{1}{4} (E_{X_i})^2 = 3 + \frac{1}{4} (-\frac{3}{4} ) = \frac{45}{16} > 0.$$  
Thus, to show that $N_i$ is ample, it suffices to check that $N_{Y_i} \cdot C_i' > 0$ for any irreducible curve $C_i' \neq G_i^-$ on $Y_i$ by the Nakai-Moishezon criterion.

 Let $L_i:=5C_i^-+24F_i-E_i$, then $L_i$ is very ample by Theorem \ref{thm:positivity on blow up ruled}. Moreover, we have
 $$L_i+2G_i^-=5C_i^-+24F_i-E_i+2(C_i^--E_i)=8N_{Y_i}.$$
 It follows that 
 $$8N_{Y_i}\cdot C_i'= (L_i+2G_i^-)\cdot C_i'>0$$
 for any irreducible curve $C_i' \neq G_i^-$ on $Y_i$.
\end{proof}
\begin{rem}\label{rem: Cartier index C_X unbdd}

The Cartier index of $E_{X_i}$ is unbounded. 
Indeed, if it were bounded, then, since the Cartier index of $C_{X_i}^+$ is bounded, it follows that the Cartier index of $N_i = C_{X_i}^+ - \frac{1}{2} E_{X_i}$ is also bounded.
Since $N_i-K_{X_i}$ is ample,   by the effective base point free theorem \cite[Theorem 1.1]{fujinoEffectiveBasePoint2009} and the very ampleness lemma \cite[Lemma 7.1]{fujinoEffectiveBasepointfreeTheorem2017} for lc pairs,
 there exists $l\in \bN$ such that $lN_i$ is very ample for all $i$. Then $\pi_i^*(8lN_i)\sim 8lC_i^+-4lE_i-lG_i^-$ is base point free on $Y_i$, 
and thus $\cO_{Y_i}(8lC_i^+-4lE_i-lG_i^-)|_{G_i^-}\simeq \cO_{G_i^-}$. 
It follows that $\cO_{C^-}(3lp_i)\simeq \cO_{\bP^2}(l)|_{C^-}\simeq \cO_{C^-}(3lp_0)$ for all $i$,  which is a contradiction.
 Similarly, we can show that if we choose $p_i\in C^-\subset W$ such that  the divisor $p_i-p_0\in \Pic^0(C^-)$ is non-torsion, then no multiple of $E_{X_i}$ is Cartier.

   Note that in our example, $Z$ is fixed and $Y_i$ is bounded. However, since the Cartier index of $E_{X_i}$ is unbounded, by the proof of \cite[Claim 2.19.1]{kollarLecturesResolutionSingularities2007}, we deduce that $X_i$ cannot be obtained from $Z$ by blowing up a zero-dimensional subscheme  of bounded length. 
   Thus, we cannot apply the "sandwich" argument in \cite[Proof of Theorem 6.9]{alexeevBoundednessLogSurfaces1994} to get the boundedness of $X_i$.

   \end{rem}
   
Now,  $M_i := 2C_{X_i}^+ - \frac{1}{2}E_{X_i}$ is an ample $\bQ$-divisor on $X_i$ with  unbounded Cartier index. Then,  
$$ M_{Y_i} := \pi_i^*M_i = 2C_i^+ - \frac{1}{2} ( E_i + \frac{1}{4}G_i^- ) = \frac{1}{8} (16C_i^+ - 4E_i - G_i^-) $$  
is nef and big, and hence $16C_i^+ - 4E_i - G_i^-$ is a nef Cartier divisor on $Y_i$. Moreover, $(Y_i, G_i^- + M_{Y_i})$ is a generalised lc pair.  
Therefore, $(X_i, M_i)$ is a generalised lc surface with data $Y_i \to X_i$ and $M_{Y_i}$.
Moreover,  $K_{X_i}+M_i\qlin C_{X_i}^+ -\frac{1}{2}E_{X_i}$ is ample and 
$$\vol(K_{X_i}+M_i)=(C_{X_i}^+)^2+\frac{1}{4}(E_{X_i})^2=\frac{45}{16}.$$
Thus,
$$(X_i,M_i)\in\cP_{-\infty}\subset \cF_{glc}(2,\Phi,v)$$
	 for $\Phi=\{0,\frac{1}{8}\}$ and $v=\frac{45}{16}$. 

\begin{thm}\label{thm:P<0 unbdd}
    The set of   underlying varieties   appearing in $\cP_{-\infty}$ is not bounded.
\end{thm}
\begin{proof}
    We will show that  there is no fixed $r \in \bN$ and a very ample divisor $H_i$ on $X_i$ satisfying $H_i^2\leq r$ for all $i$.  Thus, $X_i$ is not bounded.

    Let $H_{Y_i}:=\pi_i^*H_i$, then $H_{Y_i}$ is a big and base point free Cartier divisor on $Y_i$, which contracts only $G_i^-$.
    Since $\Pic(Y_i) = \mathbb{Z} C_i^+ \oplus \mathbb{Z} F_i \oplus \mathbb{Z} E_i$,  
write  
$$ H_{Y_i} \sim a_iC_i^+ + b_iF_i + c_iE_i, $$  
then we have  
$$ H_{Y_i} \cdot G_i^- = 0, \quad H_{Y_i} \cdot E_i > 0, \quad \text{and} \quad H_{Y_i} \cdot C_i^+ > 0, $$  
which implies  
\begin{equation}  
b_i = -c_i, \quad -c_i > 0, \quad \text{and} \quad 3a_i + b_i > 0.  
\label{eq:bi_ci_ai}
\end{equation}  
Moreover,  we have  
\begin{align}  
H_i^2 &= (a_iC_i^+ + b_iF_i - b_iE_i)^2 \notag \\  
&= (a_iC_i^+ + b_iF_i)^2 + b_i^2E_i^2 \notag \\  
&= 3a_i^2 + 2a_ib_i - b_i^2 \notag \\  
&= (a_i + b_i)(3a_i - b_i) > 0.  \label{eq:Hi_squared}
\end{align} 
Combining (\ref{eq:bi_ci_ai}) and (\ref{eq:Hi_squared}), we obtain $a_i + b_i > 0$ and $3a_i - b_i > 0$. For $H_{Y_i}$ to be base point free on $Y_i$, we must have $H_{Y_i}|_{G_i^-} \simeq \mathcal{O}_{G_i^-}$, which gives $\mathcal{O}_{Y_i}(3(a_iC_i^+ + b_iF_i - b_iE_i))|_{G_i^-} \simeq \mathcal{O}_{G_i^-}$, and hence
$$\mathcal{O}_{C^-}(3b_i (p_0 - p_i)) \simeq \mathcal{O}_{C^-}.$$ Thus, we have $b_i \to +\infty$ as $i \to +\infty$. Therefore, $H_i^2 \to +\infty$ as $i \to +\infty$.
   \end{proof}
   
\begin{rem}
    If we choose $p_i\in C^-\subset W$ such that  the divisor $p_i-p_0\in \Pic^0(C^-)$ is non-torsion, then $X_i$ is not projective.
    Indeed, if $X_i$ is projective, then there exists a very ample divisor $H_i$ on $X_i$. 
Hence, by the proof of Theorem \ref{thm:P<0 unbdd}, we have  
 $\mathcal{O}_{C^-}(3b_i (p_0 - p_i)) \simeq \mathcal{O}_{C^-}$.  
It follows that $p_i - p_0 \in \Pic(C^-)$ is torsion, which leads to a contradiction.

\end{rem}

\begin{proof}[Proof of Theorem \ref{thm: unbdd}]
    It follows from Theorem \ref{thm:P_0 unbdd}, Theorem \ref{thm:P_1 unbdd}, Theorem \ref{thm:P_2 unbdd} and Theorem \ref{thm:P<0 unbdd}
    
\end{proof}


		\bibliographystyle{alphaurl}
		\bibliography{glc-surface}
		
	\end{document}